\documentclass[11pt, a4paper, twoside]{article}
\usepackage{amsfonts}
\usepackage{mathrsfs}
\usepackage{amsmath, amsthm}
\usepackage{pgf,tikz}
\usepackage{amssymb}
\usepackage{fancyhdr}
\usepackage[colorlinks, linkcolor=black, anchorcolor=black, citecolor=black]{hyperref}

\numberwithin{equation}{section}

\allowdisplaybreaks

\begin{document}

\fancyhf{}

\fancyhead[OR]{\thepage}

\renewcommand{\headrulewidth}{0pt}
\renewcommand{\thefootnote}{\fnsymbol {footnote}}

\theoremstyle{plain} 
\newtheorem{thm}{\indent\sc Theorem}[section] 
\newtheorem{lem}[thm]{\indent\sc Lemma}
\newtheorem{cor}[thm]{\indent\sc Corollary}
\newtheorem{prop}[thm]{\indent\sc Proposition}
\newtheorem{claim}[thm]{\indent\sc Claim}
\theoremstyle{definition} 
\newtheorem{dfn}[thm]{\indent\sc Definition}
\newtheorem{rem}[thm]{\indent\sc Remark}
\newtheorem{ex}[thm]{\indent\sc Example}
\newtheorem{notation}[thm]{\indent\sc Notation}
\newtheorem{assertion}[thm]{\indent\sc Assertion}
%
%
\numberwithin{equation}{section}
\renewcommand{\proofname}{\indent\sc Proof.} 
\def\C{\mathbb{C}}
\def\R{\mathbb{R}}
\def\Rn{{\mathbb{R}^n}}
\def\M{\mathbb{M}}
\def\N{\mathbb{N}}
\def\Q{{\mathbb{Q}}}
\def\Z{\mathbb{Z}}
\def\F{\mathcal{F}}
\def\L{\mathcal{L}}
\def\S{\mathcal{S}}
\def\supp{\operatorname{supp}}
\def\essi{\operatornamewithlimits{ess\,inf}}
\def\esss{\operatornamewithlimits{ess\,sup}}
\def\dlim{\displaystyle\lim}

\fancyhf{}

\fancyhead[EC]{W. LI, H. Wang, D. Yan}

\fancyhead[EL]{\thepage}

\fancyhead[OC]{Pointwise Convergence for nonelliptic Schr\"{o}dinger Means}

\fancyhead[OR]{\thepage}

\renewcommand{\headrulewidth}{0pt}
\renewcommand{\thefootnote}{\fnsymbol {footnote}}

\title{\textbf{Sharp convergence for sequences of nonelliptic Schr\"{o}dinger means}
\footnotetext {This work is supported by the National Natural Science Foundation of China (No.11871452); Natural Natural Science Foundation of China (No.11701452); China Postdoctoral Science Foundation (No.2017M613193);  Natural Science Basic Research Plan in Shaanxi Province of China (No.2017JQ1009).}
\footnotetext {{}{2000 \emph{Mathematics Subject
 Classification}: 42B20, 42B25, 35S10.}}
\footnotetext {{}\emph{Key words and phrases}: nonelliptic Schr\"{o}dinger mean, Pointwise convergence. } } \setcounter{footnote}{0}
\author{
Wenjuan Li, Huiju Wang, Dunyan Yan}

\date{}
\maketitle

\begin{abstract}
We consider pointwise convergence of nonelliptic  Schr\"{o}dinger means $e^{it_{n}\square}f(x)$ for $f \in H^{s}(\mathbb{R}^{2})$ and decreasing sequences $\{t_{n}\}_{n=1}^{\infty}$ converging to zero, where
\[{e^{it_{n}\square }}f\left( x \right): = \int_{{\mathbb{R}^2}} {{e^{i\left( {x \cdot \xi  + t_{n}{{ \xi_{1}\xi_{2} }}} \right)}}\widehat{f}} \left( \xi  \right)d\xi .\]
We prove that when $0<s < \frac{1}{2}$,
\[\mathop {\lim }\limits_{n \to \infty} {e^{it_{n}\square }}f\left( x \right) = f(x) \hspace{0.2cm} a.e.\hspace{0.2cm} x\in \mathbb{R}^2\]
holds for all $f \in {H^s}\left( {{\mathbb{R}^2}} \right)$ if and only if $\{t_{n}\}_{n=1}^{\infty} \in \ell^{r(s), \infty}(\mathbb{N})$, $r(s)=\frac{s}{1-s}$. Moreover, our result remains valid in general dimensions.
\end{abstract}

\section{Introduction}
Consider the generalized Schr\"{o}dinger equation
\begin{equation}\label{Eq1}
\begin{cases}
 \partial_{t}u(x,t)-iP(D)u(x,t) =0 \:\:\:\ x \in \mathbb{R}^{N}, t \in \mathbb{R}^{+},\\
u(x,0)=f \\
\end{cases}
\end{equation}
where $D=\frac{1}{i}(\frac{\partial}{\partial x_{1}},\frac{\partial}{\partial x_{2}},...,\frac{\partial}{\partial x_{N}})$, $P(\xi)$ is a real continuous function defined on $\mathbb{R}^{N}$, $P(D)$ is defined via its real symbol
\[P(D)f(x)= \int_{\mathbb{R}^{N}}{e^{ix \cdot \xi}P(\xi)\hat{f}(\xi)d\xi}.\]
 The solution of (\ref{Eq1}) can be formally written as
\begin{equation}
e^{itP(D)}f(x):= \int_{\mathbb{R}^{N}}{e^{ix \cdot \xi +itP(\xi)} \hat{f} (\xi)d\xi },
\end{equation}
where $\hat{f}(\xi)$ denotes the Fourier transform of $f$. The related pointwise convergence problem is to determine the optimal $s$ for which
\begin{equation}
\mathop{lim}_{t \rightarrow 0^{+}} e^{itP(D)}f(x) = f(x)
\end{equation}
almost everywhere whenever $f \in H^{s}(\mathbb{R}^{N})$.

In the elliptic case: $P(\xi)=|\xi|^2$, the pointwise convergence problem was first considered by Carleson \cite{C} and he showed the convergence for $s\geq1/4$ when $N=1$.  Dahlberg-Kenig \cite{DK} showed that the convergence does not hold for $s<1/4$ in any dimension, which implies sharpness for the condition given by Carleson in one-dimensional case. In higher dimension $N\geq2$, Sj\"{o}lin \cite{S} and Vega \cite{V} independently obtained the convergence for $s>1/2$. In 2016, Bourgain \cite{B} gave a counterexample showing that it is false if  $s<\frac{N}{2(N+1)}$. Recently,  Du-Guth-Li \cite{DGL} for $N=2$ and Du-Zhang \cite{DZ} for higher dimensions $N\geq3$ obtained the sharp result for convergence up to the endpoint. Moreover, recent progress for the fractional Schr\"{o}dinger operators when $P(\xi)=|\xi|^\alpha$, $\alpha >1$ can be found in \cite{CK}.

Another interesting case is  the nonelliptic case: $P(\xi)=\xi_1^2-\xi_2^2\pm \xi_3^2\pm \cdots \pm \xi_N^2$.  for physical application of the nonelliptic Schr\"{o}dinger equation, see for example \cite{SPS}. Rogers-Vargas-Vega \cite{RVV} showed that the pointwise convergence of the solution to the nonelliptic Schr\"{o}dinger equation, $i\partial_tu+(\partial_x^2-\partial_y^2)u=0$, was proved when $f\in H^s(\mathbb{R}^2)$, if and only if $s\geq 1/2$. Thus the pointwise behavior is worse than that in the elliptic case. In higher dimensions, they also established similar results except the endpoint.

One of the natural generalizations of the pointwise convergence problem is to ask a.e. convergence of the Schr\"{o}dinger means where the limit is taken over decreasing sequences  $\{t_{n}\}_{n=1}^{\infty}$ converging to zero. That is to investigate relationship between optimal $s$  and properties of $\{t_{n}\}_{n=1}^{\infty}$  such that for each  function $f \in H^{s}(\mathbb{R}^{N})$,
 \begin{equation}\label{Eq1.3}
 \lim_{n \rightarrow \infty}e^{it_{n}P(D)}f(x) = f(x) \hspace{0.2cm} a.e.\hspace{0.2cm} x\in \mathbb{R}^N.
 \end{equation}
This problem was first considered by Sj\"{o}lin \cite{S1} in general dimensions and later improved by Sj\"{o}lin-Str\"{o}mberg \cite{SS1} for $P(\xi) = |\xi|^{\alpha}$, $\alpha >1$. Dimou-Seeger \cite{DS} obtained a sharp characterization of this problem in the one-dimensional case for $P(\xi) = |\xi|^{\alpha}$, $\alpha >0$. More recently, Li-Wang-Yan \cite{LW} improved the previous results of Sj\"{o}lin \cite{S1} and Sj\"{o}lin-Str\"{o}mberg \cite{SS1} for  $P(\xi) = |\xi|^{2}$ in $\mathbb{R}^{2}$ through the bilinear method.

In this paper, we concentrate ourselves on the nonelliptic case and seek what happens if $0<s<\frac{1}{2}$. More concretely, we obtain a sufficient and necessary condition for $\{t_{n}\}_{n=1}^{\infty}$ to ensure a.e. convergence of nonelliptic Schr\"{o}dinger means. For convenience, we first set $N=2$. By changing of variables, the nonelliptic Schr\"{o}dinger operator can be written as
\begin{equation}
e^{it\square}f(x):= \int_{\mathbb{R}^{2}}{e^{ix \cdot \xi +it\xi_{1}\xi_{2}} \hat{f} (\xi)d\xi }.
\end{equation}
In what follows, we always assume that the decreasing sequence $\{t_{n}\}_{n=1}^{\infty}$  converges to zero and $\{t_{n}\}_{n=1}^{\infty} \subset (0,1)$.
In order to characterize the convergence of  $\{t_{n}\}_{n=1}^{\infty}$,  we introduce the Lorentz space ${\ell}^{r,\infty}(\mathbb{N})$, $r>0$. The sequence $\{t_{n}\}_{n=1}^{\infty} \in {\ell}^{r,\infty}(\mathbb{N})$ if and only  if
\begin{equation}
\mathop{sup}_{b>0}b^{r}\sharp\biggl\{n:t_{n}>b\biggl\} < \infty.
\end{equation}
Our main results are as follows.

\begin{thm}\label{theorem1.2}
Let $0< s <\frac{1}{2}$,  $t_{n} -t_{n+1}$  be decreasing.
Then
\[\mathop {\lim }\limits_{n \to \infty} {e^{it_{n}\square }}f\left( x \right) = f(x) \hspace{0.2cm} a.e.\hspace{0.2cm} x\in \mathbb{R}^2\]
holds for all $f \in H^{s}(\mathbb{R}^{2})$ if and only if $\{t_{n}\}_{n=1}^{\infty} \in {\ell}^{r(s),\infty}(\mathbb{N})$, $r(s)= \frac{s}{1-s}$.
\end{thm}

Theorem \ref{theorem1.2} provides  a sharp  condition of $\{t_{n}\}_{n=1}^{\infty}$ for convergence of nonelliptic Schr\"{o}dinger means to hold. The sufficient and necessary conditions in Theorem \ref{theorem1.2} will be proved in Section 2 and Section 3, respectively.  The proof of the sufficient condition depends heavily on the following Theorem \ref{theorem1.4}.

\begin{thm}\label{theorem1.4}
If supp $\hat{f} \subset  \{\xi : |\xi| \sim \lambda\}$, $\lambda \ge 1$,  then for any small   interval $I$ with
\[\lambda^{-2} \le |I| \le \lambda^{-1},\]
we have
 \begin{equation}\label{Eq1.8+}
\biggl\|\mathop{sup}_{t \in I}|e^{it\square}f(x)|\biggl\|_{L^2(B(0,1))} \leq C\lambda |I|^{\frac{1}{2}}\|f\|_{L^{2}},
\end{equation}
where the constant $C$ does not depend on $f$.
\end{thm}

The proof of Theorem \ref{theorem1.4} is not hard. Indeed, it follows from Sobolev's embedding and Plancherel theorem that
 \begin{align}
 &\biggl\|\mathop{sup}_{t \in I}|e^{it\square}f| \biggl\|_{L^2(B(0,1))} \nonumber\\
 &\le \|f\|_{L^2} +  \biggl\|  \int_{{\mathbb{R}^2}} {{e^{i(x \cdot \xi + t\xi_{1}\xi_{2} )}}\widehat{f}} \left( \xi  \right)d\xi  \biggl\|^{1/2}_{L^2(B(0,1)\times I)} \nonumber\\
 & \:\  \times  \biggl\|  \int_{{\mathbb{R}^2}} {{e^{i( x\cdot \xi + t\xi_{1}\xi_{2} )}} \xi_{1}\xi_{2}\widehat{f}} \left( \xi  \right)d\xi  \biggl\|^{1/2}_{L^2(B(0,1)\times I)} \nonumber\\
 &\le \|f\|_{L^2} + |I|^{\frac{1}{2}} \|\widehat{f}\|_{L^{2}}^{1/2} \| \xi_{1}\xi_{2}\widehat{f}(\xi)\|_{L^{2}}^{1/2} \nonumber\\
 &\le \|f\|_{L^2} + \lambda|I|^{\frac{1}{2}} \|f\|_{L^{2}} \nonumber\\
 &\leq \lambda|I|^{\frac{1}{2}} \|f\|_{L^{2}}. \nonumber
 \end{align}
Then we arrive at inequality (\ref{Eq1.8+}).

Rogers-Vargas-Vega \cite{RVV} applied the stationary phase method to show the sharp estimate
 \begin{equation}\label{Eq1.8}
\biggl\|\mathop{sup}_{t \in (0,1)}|e^{it\square}f(x)|\biggl\|_{L^2(B(0,1))} \leq C\lambda^{\frac{1}{2} }\|f\|_{L^{2}(\mathbb{R}^2)},\hspace{0.2cm}\textmd{supp} \hat{f} \subset  \{\xi : |\xi| \sim \lambda\}.
\end{equation}
This implies that if $|I|=\lambda^{-1}$, then
 \begin{equation}\label{Eq1.9}
\biggl\|\mathop{sup}_{t \in I}|e^{it\square}f(x)|\biggl\|_{L^2(B(0,1))} \leq C\lambda^{\frac{1}{2} }\|f\|_{L^{2}(\mathbb{R}^2)},
\end{equation}
which coincides with Theorem \ref{theorem1.4} when $|I|=\lambda^{-1}$.
Due to the localizing lemma in Remark 3.1 of Lee-Rogers \cite{LR}, inequality (\ref{Eq1.9}) yields inequality (\ref{Eq1.8}). Therefore, inequality (\ref{Eq1.9}) is also sharp.

Moreover, by the same method as  we applied to prove Theorem \ref{theorem1.2}, we can get the corresponding result in general dimensions $N\geq 2$.

 \begin{thm}\label{theorem1.1}
Let $0< s <\frac{1}{2}$,  $t_{n} -t_{n+1}$  be decreasing, $P(\xi)=\xi_1^2-\xi_2^2\pm \xi_3^2\pm \cdots \pm \xi_N^2$.
Then
\[\mathop {\lim }\limits_{n \to \infty} {e^{it_{n}P(D) }}f\left( x \right) = f(x) \hspace{0.2cm} a.e.\hspace{0.2cm} x\in \mathbb{R}^N\]
holds for all $f \in H^{s}(\mathbb{R}^{N})$ if and only if $\{t_{n}\}_{n=1}^{\infty} \in {\ell}^{r(s),\infty}(\mathbb{N})$, $r(s)= \frac{s}{1-s}$.
\end{thm}

 We prefer to omit the proof of Theorem \ref{theorem1.1} since it is very similar with that of Theorem \ref{theorem1.2}. But a simple explanation for the proof of the necessary condition in Theorem \ref{theorem1.1} will be given in Section 3.

\textbf{Conventions}: Throughout this article, we shall use the well known notation $A\gg B$, which means if there is a sufficiently large constant $G$, which does not depend on the relevant parameters arising in the context in which
the quantities $A$ and $B$ appear, such that $ A\geq GB$. We write $A\sim B$, and mean that $A$ and $B$ are comparable. By
$A\lesssim B$ we mean that $A \le CB $ for some constant $C$ independent of the parameters related to  $A$ and $B$. $B(0,1)$ denotes the unit ball centered at the origin in $\mathbb{R}^{2}$.

\section{Sufficient condition}
By standard arguments, in order to obtain the convergence result, it is sufficient to
show the maximal function estimate in $\mathbb{R}^{2}$. In order to involve the endpoint $r(s)= \frac{s}{1-s}$, we adopt the similar decomposition as Proposition 2.3 in \cite{DS} to prove Theorem \ref{theorem2.1}.
\begin{thm}\label{theorem2.1}
 If $\{t_{n}\}_{n=1}^{\infty} \in {\ell}^{r(s),\infty}(\mathbb{N})$, $r(s)= \frac{s}{1-s}$.
Then for any $0<s <\frac{1}{2}$, we have
\begin{equation}\label{Eq1.7+}
\biggl\|\mathop{sup}_{n \in \mathbb{N}} |e^{it_{n}\square}f|\biggl\|_{L^{2}(B(0,1))} \leq C\|f\|_{H^s(\mathbb{R}^2)},
\end{equation}
whenever $f\in H^s(\mathbb{R}^2)$, where the constant $C$  does not depend on $f$.
\end{thm}
\begin{proof}
Set
\[s=\frac{r}{r+1}, \:\ r \in (0,1).\]
We decompose $f$ as
\[f=\sum_{k=0}^{\infty}{f_{k}},\]
where $\textmd{supp} \hat{f_{0}} \subset B(0,1)$, $\textmd{supp} \hat{f_{k}} \subset \{\xi: |\xi| \sim 2^{k}\}, k \ge 1$.

We decompose $\{t_{n}\}_{n=1}^{\infty}$ as
\[A_{l}:= \biggl\{t_{n}:  2^{-(l+1)\frac{2}{1+r}}<t_{n} \le 2^{-l\frac{2}{1+r}} \biggl\},\:\ l \in \mathbb{N}.\]
Since  $\{t_{n}\}_{n=1}^{\infty} \in {\ell}^{r,\infty}(\mathbb{N})$ and $r \in (0,1)$, we have
\begin{equation}\label{Eq2.5}
\sharp A_{l}  \le C 2^{\frac{2rl}{r+1}}.
\end{equation}
Then we have
\begin{align}
\mathop{sup}_{n \in \mathbb{N}} |e^{it_{n}\square}f|
  &= \mathop{sup}_{l \in \mathbb{N}}\mathop{sup}_{n: t_{n} \in A_{l}} \biggl|\sum_{k=0}^{\infty}e^{it_{n}\square}f_{k} \biggl| \nonumber\\
 &= \mathop{sup}_{l \in \mathbb{N}}\mathop{sup}_{n: t_{n} \in A_{l}} \biggl|\sum_{k \ge l}e^{it_{n}\square}f_{k} \biggl|
 +\mathop{sup}_{l \in \mathbb{N}}\mathop{sup}_{n: t_{n} \in A_{l}} \biggl|\sum_{\frac{l}{1+r} \le k < l}e^{it_{n}\square}f_{k} \biggl|
 \nonumber\\
& \:\:\ +\mathop{sup}_{l \in \mathbb{N}}\mathop{sup}_{n: t_{n} \in A_{l}} \biggl|\sum_{k < \frac{l}{1+r} }e^{it_{n}\square}f_{k} \biggl|
 \nonumber\\
&: =I + II +III.
\end{align}
Next, we will estimate $I,II,III$ respectively.

We firstly estimate $I$. We make the  change of variable $k=l+m$ in $I$. Inequality (\ref{Eq2.5}) and Plancherel theorem  imply  that
\begin{align}\label{Eq2.6}
\|I \|_{L^{2}(B(0,1))} &\le \sum_{m \ge 0} \biggl( \sum_{l \in \mathbb{N}} \sum_{n \in \mathbb{N}: t_{n} \in A_{l}}{ \biggl\|e^{it_{n}\square}f_{l+m}\biggl\|^{2}_{L^{2}(B(0,1))}}\biggl)^{1/2} \nonumber\\
&\le \sum_{m \ge 0} \biggl( \sum_{l \in \mathbb{N}} 2^{\frac{2rl}{r+1}}{ \|f_{l+m} \|^{2}_{L^{2}}}\biggl)^{1/2} \nonumber\\
&= \sum_{m \ge 0} 2^{-\frac{mr}{1+r}} \biggl( \sum_{l \in \mathbb{N}} 2^{\frac{2(l+m)r}{r+1}}{ \|f_{l+m} \|^{2}_{L^{2}}}\biggl)^{1/2} \nonumber\\
&\lesssim \|f\|_{H^{s}(\mathbb{R}^{2})}.
\end{align}

For $II$, we make the change of variable $k=l-j$. Then we have
\begin{align}
II&=\mathop{sup}_{l \in \mathbb{N}}\mathop{sup}_{n: t_{n} \in A_{l}} \biggl|\sum_{0< j \le \frac{rl}{1+r} }e^{it_{n}\square}f_{l-j} \biggl| \nonumber\\
&\le \sum_{j \in \mathbb{N}}\mathop{sup}_{l \in \mathbb{N}: l \ge \frac{r+1}{r}j}\mathop{sup}_{n: t_{n} \in A_{l}} \biggl| e^{it_{n}\square}f_{l-j} \biggl|. \nonumber
\end{align}

Therefore, we have
\begin{align}\label{Eq2.5+}
\|II \|_{L^{2}(B(0,1))} &\le \sum_{j \ge 0} \biggl( \sum_{l \in \mathbb{N}: l \ge \frac{r+1}{r}j} { \biggl\|\mathop{sup}_{n: t_{n} \in A_{l}} | e^{it_{n}\square}f_{l-j} |\biggl\|^{2}_{L^{2}(B(0,1))}}\biggl)^{1/2} \nonumber\\
&\le \sum_{j \ge 0} \biggl( \sum_{l \in \mathbb{N}: l \ge \frac{r+1}{r}j} 2^{2(l-j)} 2^{-\frac{2l}{r+1}}{ \|f_{l-j} \|^{2}_{L^{2}}}\biggl)^{1/2} \nonumber\\
&= \sum_{j \ge 0} 2^{-\frac{j}{1+r}} \biggl( \sum_{l \in \mathbb{N}: l \ge \frac{r+1}{r}j} 2^{\frac{2r(l-j)}{r+1}}{ \|f_{l-j} \|^{2}_{L^{2}}}\biggl)^{1/2} \nonumber\\
&\lesssim \|f\|_{H^{s}(\mathbb{R}^{2})},
\end{align}
where we used Theorem \ref{theorem1.4} to obtain
\begin{equation}
{ \biggl\|\mathop{sup}_{n: t_{n} \in A_{l}} | e^{it_{n}\square}f_{l-j} |\biggl\|_{L^{2}(B(0,1))}} \le 2^{l-j} 2^{-\frac{l}{r+1}}{ \|f_{l-j} \|_{L^{2}}},
\end{equation}
since
\[2^{-2(l-j)} \le 2^{-\frac{2l}{1+r}} \le 2^{-(l-j)} .\]

Finally we estimate $III$. We have
\begin{align}
III&=\mathop{sup}_{l \in \mathbb{N}}\mathop{sup}_{n: t_{n} \in A_{l}} \biggl|\sum_{k < \frac{\l}{1+r} }e^{it_{n}\square}f_{k} \biggl| \nonumber\\
&\le \sum_{k \in \mathbb{N}}\mathop{sup}_{l \in \mathbb{N}: l > (r+1)k}\mathop{sup}_{n: t_{n} \in A_{l}} \biggl| e^{it_{n}\square}f_{k} \biggl|. \nonumber
\end{align}
Notice that when $l \in \mathbb{N}$, $A_{l} \subset (0,2^{-2k})$, then we have
\begin{align}\label{Eq2.7+}
\|III \|_{L^{2}(B(0,1))} &\le \sum_{k \ge 0}  { \biggl\| \mathop{sup}_{l \in \mathbb{N}: l \ge (r+1)k} \mathop{sup}_{n: t_{n} \in A_{l}} | e^{it_{n}\square}f_{k} |\biggl\|_{L^{2}(B(0,1))}}  \nonumber\\
&\le \sum_{k \ge 0}   { \biggl\|\mathop{sup}_{ t \in (0,2^{-2k})} | e^{it\square}f_{k} |\biggl\|_{L^{2}(B(0,1))}}  \nonumber\\
&= \sum_{k \ge 0}   \|f_{k} \|_{L^{2}} \nonumber\\
&\le \|f\|_{H^{s}(\mathbb{R}^{2})}.
\end{align}

Combining  (\ref{Eq2.6}), (\ref{Eq2.5+}) and (\ref{Eq2.7+}),  inequality (\ref{Eq1.7+}) holds true for all $f \in H^{s}$.

\end{proof}

\section{Necessary Condition}
By Nikishin's theorem, the weak type estimate (\ref{Eq3.1+}) can be established from the pointwise convergence result. One can see the appendix in \cite{DS} for more details. Then the necessary condition in Theorem \ref{theorem1.2} can be obtained from the following theorem.
\begin{thm}\label{theorem3.1}
Let $0< s <\frac{1}{2}$, $t_{n} -t_{n+1}$ be decreasing. If the weak type estimate
\begin{equation}\label{Eq3.1+}
\biggl|\biggl\{x \in B(0,1): \mathop{sup}_{n \in \mathbb{N}} |e^{it_{n}\square}f| > \frac{1}{2} \biggl\} \biggl| \le C\|f\|^{2}_{H^{s}}
\end{equation}
holds for all $f \in H^{s}(\mathbb{R}^{2})$, then $\{t_{n}\}_{n=1}^{\infty} \in \ell^{r(s), \infty}(\mathbb{N})$, $r(s)= \frac{s}{1-s}$.
\end{thm}
\begin{proof}
Since
\[1-r(s) = \frac{1-2s}{1-s} >0,\]
according to  Lemma 3.2 in \cite{DS}, if $\{t_{n}\}_{n=1}^{\infty} \notin \ell^{r(s), \infty}(\mathbb{N})$, then we can choose $\{b_{j}\}_{j=1}^{\infty}$ and $\{M_{j}\}_{j=1}^{\infty}$ satisfying
\[ \lim_{j \rightarrow \infty}b_{j} = 0, \:\ \lim_{j \rightarrow \infty}M_{j} = \infty, \]
and
\begin{equation}\label{Eq3.2}
M_{j}b_{j}^{1-r(s)} \le 1,
\end{equation}
such that
\begin{equation}\label{Eq3.2+}
\sharp\biggl\{n: b_{j} < t_{n} \le 2b_{j}\biggl\} \ge M_{j}b_{j}^{-r(s)}.
\end{equation}
By the similar argument as in \cite{DS}, Proposition 3.3, when $t_{n} \le b_{j}$, we have
\begin{equation}\label{Eq3.3}
t_{n}-t_{n+1} \le 2M_{j}^{-1}b_{j}^{r(s)+1}.
\end{equation}

For fixed $j$, choose
\[\lambda_{j}=\frac{1}{1000} M_{j}^{\frac{1}{2}}b_{j}^{-\frac{r(s)+1}{2}},\]
\[\widehat{f_{j}}(\xi_{1},\xi_{2}) =\frac{1}{\lambda_{j}} \chi_{[0,\lambda_{j}]\times[{-\lambda_{j}-1,-\lambda_{j}}]}(\xi_{1},\xi_{2}).\]
Therefore,
\begin{equation}\label{Eq3.4}
\|f_{j}\|_{H^{s}}=\lambda_{j}^{s-\frac{1}{2}}.
\end{equation}
Set
\[U_{j}=(0,\frac{\lambda_{j}b_{j}}{2}) \times (-\frac{1}{1000}, \frac{1}{1000}).\]
Notice that $U_{j} \subset B(0,1)$ due to inequality (\ref{Eq3.2}). We will show that for each $x \in U_{j}$,
\begin{equation}\label{Eq3.5}
\mathop{sup}_{n \in \mathbb{N}} |e^{it_{n}\square}f_{j}| > \frac{1}{2}.
\end{equation}

Indeed, changing of variables shows that for each $n \in \mathbb{N}$,
\begin{align}\label{Eq3.6}
|e^{it_{n}\square}f_{j}(x)|&=\biggl|\frac{1}{\lambda_{j}}\int_{-\lambda_{j}-1}^{-\lambda_{j}}\int_{0}^{\lambda_{j}}{e^{ix_{1}\xi_{1}+ix_{2}\xi_{2}+it_{n}\xi_{1}\xi_{2}}d\xi_{1}d\xi_{2}}\biggl| \nonumber\\
&=\biggl|\int_{-1}^{0}\int_{0}^{1}{e^{i\lambda_{j}(x_{1}-\lambda_{j}t_{n})\eta_{1}+ix_{2}\eta_{2}+it_{n}\lambda_{j}\eta_{1}\eta_{2}}d\eta_{1}d\eta_{2}}\biggl|.
\end{align}
For each $x \in U_{j}$, there exists a unique $n(x,j)$ such that
\[x_{1} \in (\lambda_{j}t_{n(x,j)+1}, \lambda_{j}t_{n(x,j)}].\]
It is obvious that $t_{n(x,j)} \le b_{j}$ due to $t_{n(x,j)+1} \le \frac{b_{j}}{2}$, inequality (\ref{Eq3.2+}) and the assumption that $ t_{n} -t_{n+1}$ is decreasing. Then it follows from inequality (\ref{Eq3.3}) that
\begin{equation}
|\lambda_{j}(x_{1}-\lambda_{j}t_{n(x,j)})\eta_{1}| \le 2\lambda_{j}^{2} M^{-1}_{j}b_{j}^{r(s)+1} \le \frac{1}{1000}.
\end{equation}
Also,
\begin{equation}
| x_{2}\eta_{2}|  \le \frac{1}{1000},
\end{equation}
and by inequality (\ref{Eq3.2}), we have
\begin{equation}
|\lambda_{j}t_{n(x,j)}\eta_{1}\eta_{2}|  \le \lambda_{j}b_{j}  \le \frac{1}{1000}.
\end{equation}
Therefore, if we take $n=n(x,j)$ in inequality (\ref{Eq3.6}), then the phase function will be sufficiently small such
that
\[|e^{it_{n(x,j)}\square}f_{j}(x)| > \frac{1}{2}\]
for each $x \in U_{j}$, which implies inequality (\ref{Eq3.5}).

By the weak type estimate  we have
\[\biggl|\biggl\{x \in U_{j}: \mathop{sup}_{n \in \mathbb{N}} |e^{it_{n}\square}f_{j}| > \frac{1}{2} \biggl\} \biggl| \le C\|f_{j}\|^{2}_{H^{s}}.\]
Then it follows from inequality (\ref{Eq3.4}) and inequality (\ref{Eq3.5}) that
\[1 \le C M_{j}^{s-1}.\]
This is not true when $j$ is sufficiently large.
\end{proof}

\begin{rem}
The original idea we adopted to construct the counterexample in the proof of Theorem \ref{theorem3.1} comes from \cite{RVV}. The same idea remains valid in general dimensions. For example, in $\mathbb{R}^{3}$, by changing variables, we can write
\[e^{itP(D)}f(x):= \int_{\mathbb{R}^{3}}{e^{ix \cdot \xi +it(\xi_{1}\xi_{2}\pm \xi_{3}^{2})} \hat{f} (\xi)d\xi }.\]
In order to prove the necessary condition, we only need to take
\[U_{j}=(0,\frac{\lambda_{j}b_{j}}{2}) \times (-\frac{1}{1000}, \frac{1}{1000}) \times (-\frac{1}{1000}, \frac{1}{1000})\]
and
\[\widehat{f_{j}}(\xi_{1},\xi_{2}, \xi_{3}) =\frac{1}{\lambda_{j}} \chi_{[0,\lambda_{j}]\times[{-\lambda_{j}-1,-\lambda_{j}}] \times (0,1)}(\xi_{1},\xi_{2}, \xi_{3}).\]
\end{rem}


\begin{flushleft}
\vspace{0.3cm}\textsc{Wenjuan Li\\School of Mathematics and Statistics\\Northwest Polytechnical University\\710129\\Xi'an, People's Republic of China}

\vspace{0.3cm}\textsc{Huiju Wang\\School of Mathematics Sciences\\University of Chinese Academy of Sciences\\100049\\Beijing, People's Republic of China}

\vspace{0.3cm}\textsc{Dunyan Yan\\School of Mathematics Sciences\\University of Chinese Academy of Sciences\\100049\\Beijing, People's Republic of China}

\end{flushleft}

\end{document}